\documentclass[reqno]{amsart}
\usepackage{amssymb,setspace}
\usepackage{ifpdf}
\ifpdf
 \usepackage[hyperindex]{hyperref}
\else
 \expandafter\ifx\csname dvipdfm\endcsname\relax
 \usepackage[hypertex,hyperindex]{hyperref}
 \else
 \usepackage[dvipdfm,hyperindex]{hyperref}
 \fi
\fi
\allowdisplaybreaks[4]
\numberwithin{equation}{section}
\theoremstyle{plain}
\newtheorem{thm}{Theorem}[section]
\newtheorem{prop}{Proposition}[section]
\newtheorem{lem}{Lemma}[section]
\theoremstyle{definition}
\newtheorem{dfn}{Definition}[section]
\theoremstyle{remark}
\newtheorem{rem}{Remark}[section]
\DeclareMathOperator{\td}{d\mspace{-2mu}}
\newcommand{\cmdeg}[1]{\sideset{}{_\mathrm{cm}^{#1}}\deg}

\begin{document}

\title[Integral representation of Bernoulli numbers]
{An integral representation and properties of Bernoulli numbers of the second kind}

\author[F. Qi]{Feng Qi}
\address{School of Mathematics and Informatics, Henan Polytechnic University, Jiaozuo City, Henan Province, 454010, China}
\email{\href{mailto: F. Qi <qifeng618@gmail.com>}{qifeng618@gmail.com}, \href{mailto: F. Qi <qifeng618@hotmail.com>}{qifeng618@hotmail.com}, \href{mailto: F. Qi <qifeng618@qq.com>}{qifeng618@qq.com}}
\urladdr{\url{http://qifeng618.wordpress.com}}



\begin{abstract}
In the paper, the author establishes an integral representation and properties of Bernoulli numbers of the second kind and reveals that the generating function of Bernoulli numbers of the second kind is a Bernstein function on $(0,\infty)$.
\end{abstract}

\thanks{This paper was typeset using \AmS-\LaTeX}

\maketitle

\section{Introduction}

The Bernoulli numbers of the second kind $b_0,b_1,b_2,\dotsc,b_n,\dotsc$ are defined by
\begin{equation}\label{bernoulli-second-dfn}
\frac{x}{\ln(1+x)}=\sum_{n=0}^\infty b_nx^n.
\end{equation}
They are also known as Cauchy numbers, Gregary coefficients, or logarithmic numbers. The first few Bernoulli numbers are
\begin{equation}
\begin{aligned}\label{Berno-5-values}
b_0&=1, & b_1&=\frac12, & b_2&=-\frac1{12}, & b_3&=\frac1{24}, & b_4&=-\frac{19}{720}, & b_5&=\frac3{160}.
\end{aligned}
\end{equation}
Can one establish an explicit formula for computing $b_n$ for $n\in\mathbb{N}$?
\par
In~\cite{2rd-Bernoulli-2012.tex}, by establishing an explicit formula for the $n$-th derivative of $\frac1{\ln x}$, an explicit formula for calculating $b_n$ was obtained as follows.

\begin{thm}[\cite{2rd-Bernoulli-2012.tex}]\label{Bernulli-2rd=thm}
For $n\ge2$, Bernoulli numbers of the second kind $b_n$ can be computed by
\begin{equation}\label{Bernulli-2rd=formula}
b_n=(-1)^n\frac1{n!}\Biggl(\frac1{n+1}+\sum_{k=2}^{n} \frac{a_{n,k}-na_{n-1,k}}{k!}\Biggr),
\end{equation}
where $a_{n,k}$ are defined by
\begin{equation}\label{a=n=0}
a_{n,2}=(n-1)!
\end{equation}
and
\begin{equation}\label{a=n=i=eq}
a_{n,i}=(i-1)!(n-1)!\sum_{\ell_1=1}^{n-1} \frac1{\ell_1}\sum_{\ell_2=1}^{\ell_1-1}\frac1{\ell_2}\dotsm \sum_{\ell_{i-3}=1}^{\ell_{i-4}-1}\frac1{\ell_{i-3}} \sum_{\ell_{i-2}=1}^{\ell_{i-3}-1}\frac1{\ell_{i-2}}
\end{equation}
for $n+1\ge i\ge3$.
\end{thm}

\par
In this paper, we will establish an integral representation and properties for Bernoulli numbers of the second kind $b_n$ and show that the generating function $\frac{x}{\ln(1+x)}$ in the left-hand side of~\eqref{bernoulli-second-dfn} is a Bernstein function on $(0,\infty)$.

\section{Some definitions, notions, and properties}

We first collect some necessary definitions and notations.

\begin{dfn}[\cite{mpf-1993, widder}]
A function $f$ is said to be completely monotonic on an interval $I$ if $f$ has derivatives of all orders on $I$ and
\begin{equation}
(-1)^{k-1}f^{(k-1)}(t)\ge0
\end{equation}
for $x \in I$ and $k \in\mathbb{N}$.
\end{dfn}

\begin{dfn}[\cite{compmon2, minus-one}]
A function $f$ is said to be logarithmically completely monotonic on an interval $I$ if its logarithm $\ln f$ satisfies
\begin{equation}
(-1)^k[\ln f(t)]^{(k)}\ge0
\end{equation}
for $k\in\mathbb{N}$ on $I$.
\end{dfn}

\begin{dfn}[\cite{Schilling-Song-Vondracek-2010, widder}]\label{Bernstein-funct-dfn}
A function $f:I\subseteq(-\infty,\infty)\to[0,\infty)$ is called a Bernstein function on $I$ if $f(t)$ has derivatives of all orders and $f'(t)$ is completely monotonic on $I$.
\end{dfn}

\begin{dfn}[{\cite[p.~19, Definition~2.1]{Schilling-Song-Vondracek-2010}}]
A Stieltjes function is a function $f:(0,\infty)\to[0,\infty)$ which can be written in the form
\begin{equation}\label{dfn-stieltjes}
f(x)=\frac{a}x+b+\int_0^\infty\frac1{s+x}{\td\mu(s)},
\end{equation}
where $a,b\ge0$ are nonnegative constants and $\mu$ is a nonnegative measure on $(0,\infty)$ such that
\begin{equation*}
\int_0^\infty\frac1{1+s}\td\mu(s)<\infty.
\end{equation*}
\end{dfn}

\begin{dfn}[{\cite[Definition~1]{psi-proper-fraction-degree-two.tex}}]\label{x-degree-dfn}
Let $f(t)$ be a function defined on $(0,\infty)$ and have derivatives of all orders. A number $r\in\mathbb{R}\cup\{\pm\infty\}$ is said to be the completely monotonic degree of $f(t)$ with respect to $t\in(0,\infty)$ if $t^rf(t)$ is a completely monotonic function on $(0,\infty)$ but $t^{r+\varepsilon}f(t)$ is not for any positive number $\varepsilon>0$.
\end{dfn}

We remark that Definition~\ref{x-degree-dfn} slightly but essentially modifies~\cite[Definition~1.5]{Koumandos-Pedersen-09-JMAA}.
\par
These classes of functions have the following properties and relations.

\begin{prop}[{\cite[p.~161, Theorem~12b]{widder}}]\label{Bernstein-Widder-Theorem-12b}
A necessary and sufficient condition that $f(x)$ should be completely monotonic for $0<x<\infty$ is that
\begin{equation} \label{berstein-1}
f(x)=\int_0^\infty e^{-xt}\td\alpha(t),
\end{equation}
where $\alpha(t)$ is non-decreasing and the integral converges for $0<x<\infty$.
\end{prop}

\begin{prop}[{\cite[p.~15, Theorem~3.2]{Schilling-Song-Vondracek-2010}}]
A function $f:(0,\infty)\to\mathbb{R}$ is a Bernstein function if and only if it admits the representation
\begin{equation}\label{Bernstein-integral-eq}
f(x)=a+bx+\int_0^\infty\bigl(1-e^{-xt}\bigr)\td\mu(t),
\end{equation}
where $a,b\ge0$ and $\mu$ is a measure on $(0,\infty)$ satisfying
\begin{equation*}
\int_0^\infty\min\{1,t\}\td\mu(t)<\infty.
\end{equation*}
\end{prop}

\begin{prop}[{\cite{CBerg, absolute-mon-simp.tex, compmon2, minus-one}}] \label{L-C=relation}
Any logarithmically completely monotonic function must be completely monotonic.
\end{prop}

\begin{prop}[\cite{CBerg}]
The set of all Stieltjes functions is a subset of of logarithmically completely monotonic functions on $(0,\infty)$.
\end{prop}

\begin{prop}[{\cite[pp.~161\nobreakdash--162, Theorem~3]{Chen-Qi-Srivastava-09.tex} and~\cite[p.~45, Proposition~5.17]{Schilling-Song-Vondracek-2010}}] \label{B-L=relation}
The reciprocal of any Bernstein function is logarithmically completely monotonic.
\end{prop}

For the history and survey of the notion ``logarithmically completely monotonic function'', please refer to~\cite[p.~2154, Remark~8]{subadditive-qi-guo-jcam.tex}, \cite[Introduction]{SCM-2012-0142.tex}, \cite[Remark~4.8]{Open-TJM-2003-Banach.tex}, and a lot of closely related references therein.
\par
For convenience, the notation
$\cmdeg{t}[f(t)]$
was designed in~\cite{psi-proper-fraction-degree-two.tex} to stand for the completely monotonic degree $r$ of $f(t)$ with respect to $t\in(0,\infty)$.
\par
It is obvious that the completely monotonic degree of any non-trivial Bernstein function on $(0,\infty)$ is greater than $0$.
\par
Bernstein functions have the following properties.

\begin{thm}\label{degree-Bernstein=-1-thm}
Let $f(x)$ be a Bernstein function on $(0,\infty)$. Then
\begin{equation}
\cmdeg{x}[f(x)]\ge-1.
\end{equation}
In other words, the completely monotonic degree of any Bernstein function on $(0,\infty)$ is not less than $-1$.
\end{thm}

\begin{proof}
Differentiating on both sides of the formula~\eqref{Bernstein-integral-eq} gives
\begin{equation*}
f'(x)=b+\int_0^\infty e^{-xt}t\td\mu(t).
\end{equation*}
By Definition~\ref{Bernstein-funct-dfn}, the derivative $f'(x)$ is a completely monotonic function on $(0,\infty)$. In virtue of Proposition~\ref{Bernstein-Widder-Theorem-12b}, it is derived that the measure $\mu(t)$ in~\eqref{Bernstein-integral-eq} is non-decreasing on $(0,\infty)$.
Dividing by $x$ on both sides of the formula~\eqref{Bernstein-integral-eq} leads to
\begin{equation}\label{Bernstein-divide-by-x}
\frac{f(x)}x=\frac{a}x+b+\int_0^\infty q(xt)t\td\mu(t),
\end{equation}
where
\begin{equation*}
q(u)=\frac{1-e^{-u}}u
\end{equation*}
for $u\in(0,\infty)$. By some closely related knowledge in the papers~\cite{emv-log-convex-simple.tex, mon-element-exp-final.tex, Guo-Qi-Filomat-2011-May-12.tex, mon-element-exp-gen.tex, pams-62, steffensen-pair-Anal, qcw, onsp, jmaa-ii-97, AMSPROC.TEX, best-constant-one-simple-real.tex}, we find that
\begin{equation*}
q(u)=\int_{1/e}^1s^{u-1}\td s
\end{equation*}
and
\begin{equation*}
q^{(i)}(u)=\int_{1/e}^1(\ln s)^{i}s^{u-1}\td s,
\end{equation*}
which implies that the function $q(u)$ is completely monotonic on $(0,\infty)$.
Consequently, we have
\begin{equation*}
\frac{\td^{\,i}q(xt)}{\td x^i}=t^iq^{(i)}(xt),
\end{equation*}
which means that the function $q(xt)$ is completely monotonic with respect to $x\in(0,\infty)$. Accordingly, the very right term in~\eqref{Bernstein-divide-by-x} is a completely monotonic function of $x$. As a result, the function $x^{-1}f(x)$ is completely monotonic on $(0,\infty)$, that is,
\begin{equation*}
\cmdeg{x}[f(x)]\ge-1.
\end{equation*}
The proof of Theorem~\ref{degree-Bernstein=-1-thm} is complete.
\end{proof}

\begin{thm}
The completely monotonic degree of the reciprocal of a Bernstein function on $(0,\infty)$ is non-negative and less than $1$.
\end{thm}

\begin{proof}
Let $f(x)$ be a Bernstein function on $(0,\infty)$. Then, by Theorem~\ref{degree-Bernstein=-1-thm}, the function $x^{-1}f(x)$ is completely monotonic, so its reciprocal $\frac{x}{f(x)}$ is non-decreasing, that is, the reciprocal $\frac{x}{f(x)}$ is not completely monotonic. Thus, the degree of $\frac1{f(x)}$ is less than $1$. However, by Propositions~\ref{B-L=relation} and~\ref{L-C=relation}, the reciprocal $\frac1{f(x)}$ is completely monotonic on $(0,\infty)$, so its completely monotonic degree is non-negative.
\end{proof}

\section{A lemma}

For establishing integral representations for Bernoulli numbers of the second kind $b_n$, we need the following lemma.

\begin{lem}[{\cite[p.~2130]{Berg-Pedersen-ball-PAMS}}]
The function $\frac1{\ln(1+x)}$ is a Stieltjes function and has the integral representation
\begin{equation}\label{recip-ln(1+z)}
\frac1{\ln(1+z)}=\frac1{z}+\int_1^\infty\frac1{[\ln(t-1)]^2+\pi^2} \frac{\td t}{z+t}
\end{equation}
for $z\in\mathbb{C}\setminus(-\infty,0]$.
\end{lem}

\begin{rem}
We note that the integral representation~\eqref{recip-ln(1+z)} is a correction of the equation~(34) on page~2130 in~\cite{Berg-Pedersen-ball-PAMS}.
\end{rem}

\section{Integral representations for Bernoulli numbers}

We are now in a position to establish integral representations of Bernoulli numbers of the second kind $b_k$.

\begin{thm}\label{thm-Bernoulli-Integ-Form}
The Bernoulli numbers of the second kind $b_n$ for $n\in\mathbb{N}$ may be calculated by
\begin{equation}\label{Bernoulli-2rd-int}
b_n=(-1)^{n+1}\int_1^\infty\frac{1}{\{[\ln(t-1)]^2+\pi^2\}t^{n}}\td t.
\end{equation}
\end{thm}

\begin{proof}
By~\eqref{recip-ln(1+z)}, we have
\begin{equation}\label{frac(x)(ln(1+x))-eq1}
\frac{x}{\ln(1+x)}=1+\int_1^\infty\frac1{[\ln(t-1)]^2+\pi^2} \frac{x}{x+t}\td t
\end{equation}
and
\begin{equation}
\begin{split}\label{deriv-int-expression}
\biggl[\frac{x}{\ln(1+x)}\biggr]^{(k)}&=\int_1^\infty\frac1{[\ln(t-1)]^2+\pi^2} \biggl(\frac{x}{x+t}\biggr)^{(k)}\td t\\
&=\int_1^\infty\frac1{[\ln(t-1)]^2+\pi^2} \biggl(1-\frac{t}{x+t}\biggr)^{(k)}\td t\\
&=(-1)^{k+1}k!\int_1^\infty\frac{t}{[\ln(t-1)]^2+\pi^2} \frac1{(x+t)^{k+1}}\td t
\end{split}
\end{equation}
for $k\in\mathbb{N}$. On the other hand, by~\eqref{bernoulli-second-dfn}, we also have
\begin{equation}\label{deriv-dfn-second}
\biggl[\frac{x}{\ln(1+x)}\biggr]^{(k)}=\sum_{n=k}^\infty b_n\frac{n!}{(n-k)!}x^{n-k}.
\end{equation}
Combining~\eqref{deriv-int-expression} with~\eqref{deriv-dfn-second} leads to
\begin{equation}\label{combine-two-eq-bernoulli}
\sum_{n=k}^\infty b_n\frac{n!}{(n-k)!}x^{n-k}=(-1)^{k+1}k!\int_1^\infty\frac{t}{[\ln(t-1)]^2+\pi^2} \frac1{(x+t)^{k+1}}\td t.
\end{equation}
Letting $x\to0^+$ on both sides of the above equation produces
\begin{equation*}
k!b_k=(-1)^{k+1}k!\int_1^\infty\frac1{[\ln(t-1)]^2+\pi^2} \frac1{t^{k}}\td t.
\end{equation*}
Thus, the formula~\eqref{Bernoulli-2rd-int} is proved.
\end{proof}

\section{Properties of Bernoulli numbers of the second kind}

Basing on the integral representation~\eqref{Bernoulli-2rd-int} in Theorem~\ref{thm-Bernoulli-Integ-Form} for Bernoulli numbers of the second kind $b_n$, we now turn our attention to investigate the complete monotonicity and other properties of Bernoulli numbers $b_n$ for $n\in\mathbb{N}$.
\par
We recall from~\cite[p.~108, Definition~4]{widder} that a sequence $\{\mu_n\}_0^\infty$ is said to be completely monotonic if its elements are non-negative and its successive differences are alternatively non-negative, that is
\begin{equation}
(-1)^k\Delta^k\mu_n\ge0
\end{equation}
for $n,k\ge0$, where
\begin{equation}
\Delta^k\mu_n=\sum_{m=0}^k(-1)^m\binom{k}{m}\mu_{n+k-m}.
\end{equation}
Theorem~4a in~\cite[p.~108]{widder} reads that a necessary and sufficient condition that the sequence $\{\mu_n\}_0^\infty$ should have the expression
\begin{equation}\label{mu=alpha-moment}
\mu_n=\int_0^1t^n\td\alpha(t)
\end{equation}
for $n\ge0$, where $\alpha(t)$ is non-decreasing and bounded for $0\le t\le1$, is that it should be completely monotonic.
\par
We also recall from~\cite[p.~163, Definition~14a]{widder} that a completely monotonic sequence $\{a_n\}_0^\infty$ is minimal if it ceases to be completely monotonic when $a_0$ is decreased.
Theorem~14a in~\cite[p.~164]{widder} states that a completely monotonic sequence $\{\mu_n\}_0^\infty$ is minimal if and only if the equality~\eqref{mu=alpha-moment} is valid for $n\ge0$ and $\alpha(t)$ is a non-decreasing bounded function continuous at $t=0$.

\begin{thm}\label{Ber-minimal-thm}
The sequence $\{(-1)^{n}b_{n+1}\}_{n=0}^\infty$ of Bernoulli numbers of the second kind is completely monotonic and minimal.
\end{thm}

\begin{proof}
This follows from setting in the equality~\eqref{mu=alpha-moment}
\begin{equation}
\alpha(t)=\int_0^t\frac{1}{s\{[\ln(1/s-1)]^2+\pi^2\}}\td s
\end{equation}
for $t\in[0,1]$ and $\alpha(1)=b_1=\frac12$. The proof of Theorem~\ref{Ber-minimal-thm} is complete.
\end{proof}

\begin{thm}\label{b(n)-matrix-thm}
Let $m\in\mathbb{N}$ and $a_i$ for $1\le i\le m$ be nonnegative integers. Then
\begin{equation}\label{matrix-1f}
\bigl|(a_i+a_j)!b_{a_i+a_j+1}\bigr|_m\ge0
\end{equation}
and
\begin{equation}\label{matrix-2f}
\bigl|(-1)^{a_i+a_j}{(a_i+a_j)!}b_{a_i+a_j+1}\bigr|_m\ge0,
\end{equation}
where $|a_{ij}|_m$ denotes a determinant of order $m$ with elements $a_{ij}$.
\end{thm}

\begin{proof}
From the proofs of Theorem~\ref{thm-Bernoulli-Integ-Form}, we observe that
\begin{equation}\label{b(n)-limit}
b_n=(-1)^{n+1}\lim_{x\to0^+}h_n(x),
\end{equation}
where the function
\begin{equation}\label{h(x)-dfn-eq}
h_n(x)=\int_1^\infty\frac{1}{\{[\ln(t-1)]^2+\pi^2\}(t+x)^{n}}\td t
\end{equation}
is completely monotonic on $[0,\infty)$.
\par
In~\cite{two-place}, or see~\cite[p.~367]{mpf-1993}, it was obtained that if $f$ is a completely monotonic function, then
\begin{equation}\label{matrix-eq-1}
\bigl|f^{(a_i+a_j)}(x)\bigr|_m\ge0
\end{equation}
and
\begin{equation}\label{matrix-eq-2}
\bigl|(-1)^{a_i+a_j}f^{(a_i+a_j)}(x)\bigr|_m\ge0,
\end{equation}
where $|a_{ij}|_m$ denotes a determinant of order $m$ with elements $a_{ij}$ and $a_i$ for $1\le i\le m$ are nonnegative integers. Applying $f$ in~\eqref{matrix-eq-1} and~\eqref{matrix-eq-2} to the function $h_n(x)$ yields
\begin{equation}\label{matrix-eq-1h}
\bigl|h_n^{(a_i+a_j)}(x)\bigr|_m\ge0
\end{equation}
and
\begin{equation}\label{matrix-eq-2h}
\bigl|(-1)^{a_i+a_j}h_n^{(a_i+a_j)}(x)\bigr|_m\ge0,
\end{equation}
that is,
\begin{equation}\label{matrix-eq-1b}
\biggl|(-1)^{a_i+a_j}\frac{(n+a_i+a_j-1)!}{(n-1)!}h_{n+a_i+a_j}(x)\biggr|_m\ge0
\end{equation}
and
\begin{equation}\label{matrix-eq-2b}
\biggl|\frac{(n+a_i+a_j-1)!}{(n-1)!}h_{n+a_i+a_j}(x)\biggr|_m\ge0.
\end{equation}
Letting $x\to0^+$ in~\eqref{matrix-eq-1b} and~\eqref{matrix-eq-2b} and making use of~\eqref{b(n)-limit} produce
\begin{equation}\label{matrix-1b}
\biggl|(-1)^{a_i+a_j}\frac{(n+a_i+a_j-1)!}{(n-1)!}(-1)^{n+a_i+a_j+1}b_{n+a_i+a_j}\biggr|_m\ge0
\end{equation}
and
\begin{equation}\label{matrix-2b}
\biggl|\frac{(n+a_i+a_j-1)!}{(n-1)!}(-1)^{n+a_i+a_j+1}b_{n+a_i+a_j}\biggr|_m\ge0.
\end{equation}
Further simplifying~\eqref{matrix-1b} and~\eqref{matrix-2b} leads to
\begin{equation*}
\bigl|(-1)^{n+1}{(n+a_i+a_j-1)!}b_{n+a_i+a_j}\bigr|_m\ge0
\end{equation*}
and
\begin{equation*}
\bigl|(-1)^{n+a_i+a_j+1}{(n+a_i+a_j-1)!}b_{n+a_i+a_j}\bigr|_m\ge0,
\end{equation*}
which are equivalent to~\eqref{matrix-1f} and~\eqref{matrix-2f}. Theorem~\ref{b(n)-matrix-thm} is thus proved.
\end{proof}

\begin{rem}
Taking $m=3$ and $a_i=i$ for $i=0,1,2$ in Theorem~\ref{b(n)-matrix-thm} and using values of $b_i$ for $1\le i\le5$ in~\eqref{Berno-5-values} show us that
\begin{equation*}
\begin{vmatrix}
  0!b_1 & 1!b_2 & 2!b_3 \\
  1!b_2 & 2!b_3 & 3!b_4 \\
  2!b_3 & 3!b_4 & 4!b_5
\end{vmatrix}
=
\begin{vmatrix}
  \frac12 & -\frac1{12} & \frac1{12} \\
  -\frac1{12} & \frac1{12} & -\frac{19}{120} \\
  \frac1{12} & -\frac{19}{120} & \frac35
\end{vmatrix}
=\frac{857}{86400}
\end{equation*}
and
\begin{equation*}
\begin{vmatrix}
  0!b_1 & -1!b_2 & 2!b_3 \\
  -1!b_2 & 2!b_3 & -3!b_4 \\
  2!b_3 & -3!b_4 & 4!b_5
\end{vmatrix}
=
\begin{vmatrix}
  \frac12 & \frac1{12} & \frac1{12} \\
  \frac1{12} & \frac1{12} & \frac{19}{120} \\
  \frac1{12} & \frac{19}{120} & \frac35
\end{vmatrix}
=\frac{857}{86400}.
\end{equation*}
\end{rem}

Let $\lambda=(\lambda_1,\lambda_2,\dotsc,\lambda_{n})\in\mathbb{R}^{n}$ and $\mu=(\mu_1,\mu_2,\dotsc, \mu_{n})\in\mathbb{R}^{n}$. A sequence $\lambda$ is said to be majorized by $\mu$ \textup{(}in symbols $\lambda\preceq \mu$\textup{)} if
\begin{equation*}
\sum_{i=1}^k \lambda_{[i]}\le\sum_{i=1}^k \mu_{[i]}
\end{equation*}
for $k=1,2,\dotsc,n-1$ and
$$
\sum_{i=1}^n \lambda_i=\sum_{i=1}^n\mu_i,
$$
where $\lambda_{[1]}\ge \lambda_{[2]}\ge \dotsm \ge \lambda_{[n]}$ and $\mu_{[1]}\ge \mu_{[2]}\ge\dotsm \ge \mu_{[n]}$ are rearrangements of $\lambda$ and $\mu$ in a descending order.
\par
A sequence $\lambda$ is said to strictly majorized by $\mu$ $($in symbols $\lambda \prec \mu)$ if $\lambda$ is not a permutation of $\mu$.
\par
In~\cite[p.~106, Theorem~A]{haerc1}, a correction of~\cite[Theorem~1]{finkjmaa82} which was collected in~\cite[p.~367, Theorem~2]{mpf-1993}, it was obtained that if $f$ is a completely monotonic function on $(0,\infty)$ and $\lambda\preceq\mu$, then
\begin{equation}\label{finkjmaa82=ineq3.2}
\Biggl|\prod_{i=1}^nf^{(\lambda_i)}(x)\Biggr|\le \Biggl|\prod_{i=1}^nf^{(\mu_i)}(x)\Biggr|.
\end{equation}

\begin{thm}\label{lambda-mu-thm}
Let $m\in\mathbb{N}$ and let $\lambda$ and $\mu$ be two $m$-tuples of nonnegative numbers such that $\lambda\preceq\mu$. Then
\begin{equation}\label{lambda-mu-eq}
\Biggl|\prod_{i=1}^m {\lambda_i!}b_{\lambda_i+1}\Biggr|\le \Biggl|\prod_{i=1}^m {\mu_i!}b_{\mu_i+1}\Biggr|.
\end{equation}
\end{thm}

\begin{proof}
Employing the inequality~\eqref{finkjmaa82=ineq3.2} applied to $h_n(x)$ defined by~\eqref{h(x)-dfn-eq} creates
\begin{equation*}
\Biggl|\prod_{i=1}^m(-1)^{\lambda_i} \frac{(n+\lambda_i-1)!}{(n-1)!}h_{n+\lambda_i}(x)\Biggr|
\le \Biggl|\prod_{i=1}^m (-1)^{\mu_i}\frac{(n+\mu_i-1)!}{(n-1)!}h_{n+\mu_i}(x)\Biggr|
\end{equation*}
which can be simplified as
\begin{equation*}
\Biggl|\prod_{i=1}^m {(n+\lambda_i-1)!}h_{n+\lambda_i}(x)\Biggr|
\le \Biggl|\prod_{i=1}^m {(n+\mu_i-1)!}h_{n+\mu_i}(x)\Biggr|.
\end{equation*}
Further taking $x\to0^+$ and utilizing~\eqref{b(n)-limit} turn out
\begin{equation*}
\Biggl|\prod_{i=1}^m {(n+\lambda_i-1)!}(-1)^{n+\lambda_i+1}b_{n+\lambda_i}\Biggr|\le \Biggl|\prod_{i=1}^m {(n+\mu_i-1)!}(-1)^{n+\mu_i+1}b_{n+\mu_i}\Biggr|
\end{equation*}
which is equivalent to~\eqref{lambda-mu-eq}. The proof of Theorem~\ref{lambda-mu-thm} is complete.
\end{proof}

\begin{thm}\label{b(n)-log-Convex-thm}
The sequence $\{i!b_{i+1}\}_0^\infty$ is logarithmically convex.
\end{thm}

\begin{proof}
It is clear that $(i,i+2)\succ(i+1,i+1)$ for $i\ge0$. Therefore, by virtue of~\eqref{lambda-mu-eq}, we have
\begin{equation}\label{b(n)-log-convex-eq}
(i!b_{i+1})[(i+2)!b_{i+3}]\ge [(i+1)!b_{i+2}]^2.
\end{equation}
This implies the required logarithmic convexity.
\par
This conclusion can also be deduced from Theorem~\ref{b(n)-matrix-thm}.
The proof of Theorem~\ref{b(n)-log-Convex-thm} is thus complete.
\end{proof}

\begin{rem}
Letting $i=2$ in~\eqref{b(n)-log-convex-eq} and using three corresponding values of $b_i$ for $3\le i\le5$ in~\eqref{Berno-5-values} give
\begin{equation*}
(2!b_{3})(4!b_5)=\frac3{80}=0.0375\ge(3!b_{4})^2=\frac{361}{14400}=0.0250\dotsc.
\end{equation*}
\end{rem}

\section{A Bernstein function}

Finally, we prove that the generating function $\frac{x}{\ln(1+x)}$ is a Bernstein function.

\begin{thm}\label{Bernoulli-Gen-Funct-Bernstein-thm}
The generating function $\frac{x}{\ln(1+x)}$ of Bernoulli numbers of the second kind $b_k$ is a Bernstein function on $(0,\infty)$.
\end{thm}

\begin{proof}[First proof]
The integral representation~\eqref{frac(x)(ln(1+x))-eq1} shows us that the function $\frac{x}{\ln(1+x)}$ is positive and increasing on $(0,\infty)$. The integral representation~\eqref{deriv-int-expression} reveals that the first derivative of $\frac{x}{\ln(1+x)}$ is completely monotonic on $(0,\infty)$. So, by Definition~\ref{Bernstein-funct-dfn}, the function $\frac{x}{\ln(1+x)}$ is a Bernstein function on $(0,\infty)$. The proof of Theorem~\ref{Bernoulli-Gen-Funct-Bernstein-thm} is complete.
\end{proof}

\begin{proof}[Second proof]
It is not difficult to see that
\begin{equation}
\frac{x}{\ln(1+x)}=\int_0^1(1+x)^t\td t
\end{equation}
and the function $(1+x)^t$ for $t\in(0,1)$ is a Bernstein function. Theorem~\ref{Bernoulli-Gen-Funct-Bernstein-thm} is thus proved.
\end{proof}

\end{document}